\documentclass[12pt]{amsart}
\usepackage[dvipsnames,usenames]{color}
\usepackage{hyperref}
\usepackage{graphicx}
\usepackage{epsfig}
\usepackage[latin1]{inputenc}
\usepackage{amsmath}
\usepackage{amsfonts}
\usepackage{amssymb}
\usepackage{amsthm}
\usepackage{amscd}
\usepackage{verbatim}
\usepackage{subfig}
\usepackage{pinlabel}
\usepackage{mathtools}
\usepackage{stmaryrd}
\usepackage{enumerate, enumitem}
\usepackage{array,multirow}

\usepackage[textsize=scriptsize]{todonotes}

\usepackage{tikz}
\usetikzlibrary{cd}
\usetikzlibrary{arrows}
\usetikzlibrary{decorations.pathreplacing}
\usepackage{verbatim}

\newtheorem*{theorem*}{Theorem}
\newtheorem{theorem}{Theorem}[section]
\newtheorem{lemma}[theorem]{Lemma}

\newtheorem{proposition}[theorem]{Proposition}

\theoremstyle{definition}
\newtheorem{definition}[theorem]{Definition}

\def\Z{\mathbb{Z}}

\def\F{\mathbb{F}}

\def\II {\mathcal{I}}

\def\JJ {\mathcal{J}}

\def\aa {\widetilde{\alpha}}

\def\l {\ell}
\def\co{\colon\thinspace}

\def\x {\mathbf{x}}
\DeclareMathOperator{\gr}{gr}

\def\u {\mathcal{U}}
\def\v {\mathcal{V}}

\def \DD {\Delta_{\II,\JJ}}

\def\CFi {\operatorname{CF}^\infty}

\def\HFm {\operatorname{HF}^-}

\def\CFK {\operatorname{CFK}}

\def\CFKi {\operatorname{CFK}^\infty}

\title{PL-genus of surfaces in homology balls}

\author{Jennifer Hom}
\address{School of Mathematics, Georgia Institute of Technology, Atlanta, GA, USA}
\email{hom@math.gatech.edu}
\author{Matthew Stoffregen}
\address{Department of Mathematics, Michigan State University, East Lansing, MI, USA}
\email{stoffre1@msu.edu}
\author{Hugo Zhou}
\address{School of Mathematics, Georgia Institute of Technology, Atlanta, GA, USA}
\email{hzhou@gatech.edu}

\begin{document}

\maketitle

\begin{abstract}
    We consider manifold-knot pairs $(Y,K)$ where $Y$ is a homology sphere that bounds a homology ball. We show that the minimum genus of a PL surface $\Sigma$ in a homology ball $X$ such that $\partial (X, \Sigma) = (Y, K)$ can be arbitrarily large. Equivalently, the minimum genus of a surface cobordism in a homology cobordism from $(Y, K)$ to any knot in $S^3$ can be arbitrarily large. The proof relies on Heegaard Floer homology.
\end{abstract}

\section{Introduction}\label{sec:intro}
Every knot $K$ in $S^3$ bounds a piecewise-linear (PL) disk in the 4-ball, namely, by taking the cone on the pair $(S^3, K)$. (This disk in not locally-flat, and throughout, we will not impose any local-flatness conditions on our PL surfaces). Resolving a conjecture of Zeeman \cite{Zeeman}, Akbulut \cite{Akbulut} gave an example of a contractible 4-manifold $X$ and a knot $K \subset \partial X$ such that $K$ does not bound a PL disk in $X$. However, Akbulut's $K$ does bound a PL disk in a different contractible 4-manifold $X'$ with $\partial X' = \partial X$. A. Levine \cite{Levinenonsurj} proved the stronger result that there exist manifold-knot pairs $(Y, K)$ such that $Y$ bounds a smooth, contractible 4-manifold $X$ and that $K$ does not bound a PL disk in $X$ nor in any other integer homology ball $X'$ with $\partial X' = Y$.

In light of Levine's result, a natural question to ask is: Given a knot $K$ in an integer homology 3-sphere $Y$ such that $Y$ bounds an integer homology 4-ball, what's the minimum genus of a PL surface $\Sigma$ in an integer homology ball $X$ such that $\partial (X,\Sigma) = (Y,K)$? We observe that such a surface $\Sigma$ always exists, since $K$ is null-homologous and thus bounds a surface in $Y$, which may be pushed slightly into any bounding $4$-manifold.

Our main result is that this notion of PL genus can be arbitrarily large.  Throughout, let $(Y_n, K_n) = \big(S^3_{-1}(T_{2n,2n+1}) \# -S^3_{-1}(T_{2n,2n+1}), \mu_{2n-1,-1} \# U \big)$, where $\mu_{2n-1,-1}$ denotes the $(2n-1, -1)$-cable of the meridian in $S^3_{-1}(T_{2n,2n+1})$ and $U$ denotes the unknot in $-S^3_{-1}(T_{2n,2n+1})$. 

\begin{theorem}\label{thm:main}
  Any PL surface $\Sigma$ in any integer homology ball $X$ such that $\partial(X, \Sigma) = (Y_n, K_n)$ must have genus at least $n-1$, for $n \in \Z_{>0}$.
\end{theorem}

We prove Theorem \ref{thm:main} by reinterpreting PL surfaces in terms of cobordisms inside of homology cobordisms. Recall that a \emph{homology cobordism} from $Y_0$ to $Y_1$ is smooth, compact 4-manifold $W$ such that $\partial W = Y_0 \sqcup Y_1$ and that the map $i_* \colon H_*(Y_j; \Z) \to H_*(W; \Z)$ induced by inclusion is an isomorphism for $j=0,1$. Let $\Sigma$ be a genus $g$ PL surface in an integer homology ball $X$ such that $\partial(X, \Sigma) = (Y_n, K_n)$. Up to isotopy, we may assume that $\Sigma$ is smooth, except at finitely many singular points, each of which is modeled on the cone of a smooth knot $J_i$ in $S^3$. (See, for example, \cite[Theorem A.1]{HLL}.) By deleting neighborhoods of arcs in $\Sigma$ connecting the cone points, we obtain a genus $g$ cobordism from the knot $J = J_1 \# \dots \# J_m$ to $K$ in a homology cobordism from $S^3$ to $Y$. 

Let $K$ be a knot in a homology null-bordant homology sphere $Y$. We consider cobordisms of pairs
\[(W, S) \colon (S^3, J) \to (Y, K)\]
such that $W$ is a homology cobordism from $S^3$ to $Y$. The \emph{cobordism distance} between $(Y,K)$ and $(S^3, J)$ is the minimal genus of $S$ in any such pair $(W, S)$. By the preceding discussion, Theorem \ref{thm:main} is an immediate consequence of the following result.

\begin{theorem}\label{thm:cobdist}
     The cobordism distance between $(Y_n,K_n)$ and any knot in $S^3$ is at least $n-1$. 
\end{theorem}

We prove Theorem \ref{thm:cobdist} using Heegaard Floer homology \cite{OS}, specifically Zemke's cobordism maps \cite{Zemkelinkabs}. Our obstruction relies on two key properties:

\begin{enumerate}
    \item Consider a cobordism of pairs
    \[(W, S) \colon (S^3, J) \to (Y_n, K_n)\]
    where $W$ is a homology cobordism  and $S$ has genus $g$. For any $(c_1, c_2) \in (2\Z)^2$ such that $c_1 +c_2 = -2g$ and $c_1, c_2 \leq 0$, there exists a local map
    \[ f_{W, S} \colon \CFK(S^3, J) \to \CFK(Y_n, K_n) \] 
   with bigrading $(c_1, c_2)$.  Similarly, we may consider a cobordism  in the opposite direction, from $(Y_n, K_n)$ to $(S^3, J)$. See \cite[Theorems 1.4 and 1.7]{Zemkelinkabs}.
     \item The Heegaard Floer homology of $S^3$ is especially simple; namely, $\HFm(S^3) = \F[U]$. In particular, $U$ acts nontrivally on any nontrivial element of $\HFm(S^3)$.
\end{enumerate}
See Section \ref{sec:obstruction} for more details.

For constructing our examples $(Y_n, K_n)$, we rely on recent work of the last author \cite{zhou2022filtered}, which combines work of Hedden-Levine \cite{HeddenLevine} and Truong \cite{Truong} to give a description of the knot Floer complex for $(p, 1)$-cables of the meridian in the image of surgery along a knot in $S^3$. Preliminaries on this filtered mapping cone are given in Section \ref{sec:mappingcone} and the computation is carried out in Section \ref{sec:comp}.

\subsection*{Acknowledgements}
JH and HZ were supported by NSF grant DMS-2104144 and a Simons
Fellowship. MS was supported by NSF grant DMS-1952755. This work was done while JH and HZ were in residence at the Simons
Laufer Mathematical Sciences Institute (formerly MSRI) during Fall 2023, supported by NSF Grant DMS-1928930. The authors would like to thank Irving Dai, Tye Lidman, Chuck Livingston, and Linh Truong for helpful conversations.

\section{Cobordism obstruction} \label{sec:obstruction}
In this section, we introduce a cobordism obstruction for manifold-knot pairs and prove Theorem \ref{thm:cobdist} by  applying the obstruction to the pairs $(Y_n, K_n)$, calling upon the computational results in the later part of the paper. In addition, we compute the values of the concordance homomorphisms $\varphi_{i,j}$ of \cite{Homoconcor} on the family $(Y_n, K_n)$, which may be of independent interest.

We start with some preliminaries on knot Floer homology.  Knot Floer homology was defined by Ozsv\'{a}th-Szab\'{o} \cite{OSknot} and J. Rasmussen \cite{Rasmussen}.  We associate, following the conventions of Zemke \cite{Zemkelinkabs}, to a  manifold-knot pair $(Y,K)$ a chain complex $ \CFK_{\F[\u,\v]}(Y,K) = \CFK(Y,K) $ over the polynomial ring $\F[\u,\v],$ where $\F = \Z/2\Z$, called the \textit{knot Floer complex}. This chain complex is a free module generated by intersecting points of two Lagrangians in a symmetric product of a Heegaard surface, equipped with  differentials by counting holomorphic disks, weighted over the intersection numbers with the two basepoints. The topological invariance of $\CFK(Y,K)$ up to chain homotopy equivalence over  $\F[\u,\v]$ is due to  Ozsv\'{a}th-Szab\'{o} and J. Rasmussen.

The knot Floer complex $\CFK(Y,K)$ comes with a bigrading, namely $(\gr_{\u}, \gr_{\v})$, where $\u, \v$, and $\partial$ each have  bigrading $(-2,0), (0,-2)$ and $(-1,-1)$ respectively. The Alexander grading of a homogeneous element $x\in \CFK(Y,K) $ is defined by $A(x)=\frac{1}{2}(\gr_{\u}(x) - \gr_{\v}(x))$.

A chain map between two complexes is called a \textit{local map} if it induces an isomorphism on the $(\u,\v)$-localized homology.  Following from a special case of   \cite[Thoerem 1.4]{Zemkelinkabs}, the next theorem provides the main technical input for the obstruction.

\begin{theorem}[Theorem 1.4 in \cite{Zemkelinkabs}] Suppose that  $(W,S) \co (Y_1,K_1) \rightarrow (Y_2, K_2)$ is a cobordism between the manifold-knot pairs $(Y_1,K_1)$ and $(Y_2, K_2)$, such that $W$ is a homology cobordism and $S$ is of genus $g.$ Then for any given $(c_1,c_2)\in (2\mathbb{Z})^2$ such that $c_1+c_2=-2g$ and $c_1, c_2 \leq 0$,  there exists a local map 
\[
  f_{W, S} \colon \CFK(Y_1,K_1) \to \CFK(Y_2, K_2)
\]
with bigrading $(c_1,c_2)$.
\end{theorem}

In particular, when $g(S)=0$, namely, when $(Y_1,K_1)$ and $(Y_2, K_2)$ are homology concordant, then the cobordism map  $f_{W, S}$ is a local map that preserves the bigrading. 

\begin{definition}
    Two bigraded chain complexes $C_1$ and $C_2$ over $\F[\u,\v]$ are \textit{locally equivalent} if there exist bigrading-preserving local maps 
    \begin{align*}
        f \colon C_1 \rightarrow C_2  \hspace{2em} \text{and} \hspace{2em}     g \colon C_2 \rightarrow C_1. 
    \end{align*}
\end{definition}

It is straightforward to verify that local equivalence is an equivalence relation.  By turning the cobordism around, we thus obtain that homology concordance induces local equivalence of the knot Floer complexes.   Since the cobordism distance is invariant over the homology concordance class,  we study the local equivalence class of the knot Floer complexes of the interested manifold-knot pairs.

Due to computational reasons, it is somewhat easier to first consider 
\[-(Y_n,K_n) = -\big(S^3_{-1}(T_{2n,2n+1}) \# -S^3_{-1}(T_{2n,2n+1}), \mu_{2n-1,-1} \# U \big),\]
that is, the orientation reversal of the manifold-knot pairs that appear in the Section \ref{sec:intro}. Observe that $-(S^3_{-1}(T_{2n,2n+1}), \mu_{2n-1,-1})$ is equivalent to $(S^3_{1}(-T_{2n,2n+1}), \mu_{2n-1,1})$.
According to Lemma \ref{le: localequi}, over the ring $\F[U,U^{-1}],$ the complex $X^\infty_{2n-1} (-T_{2n,2n+1}) \langle 2n-1 \rangle$ represents the local equivalence class of $\CFKi(S^3_{1}(-T_{2n,2n+1}), \mu_{2n-1,1})$ for all $n\geq 3.$  (See the beginning of Section \ref{sec:mappingcone} for more about the knot Floer complex $\CFKi(Y,K)$ defined over the ring $\F[U,U^{-1}]$.) 

Recall that the knot Floer complex enjoys a K\"{u}nneth principle by \cite[Theorem 7.1]{OSknot}. Since $-(Y_n,K_n) = \big(S^3_{1}(-T_{2n,2n+1}),\mu_{2n-1,1}\big) \# \big(S^3_{-1}(T_{2n,2n+1}), U \big) $, the knot Floer complex of the pair $-(Y_n,K_n)$ is locally equivalent to $\CFKi(S^3_{1}(-T_{2n,2n+1}), \mu_{2n-1,1})$ tensored with a trivial complex, with the Maslov grading adjusted such that the tensored complex has $d$--invariant equal to $0.$  

  Translate this into the ring $\F[\u,\v]$; for $n\geq 1,$ let $C_n$ denote the complex corresponding to $X^\infty_{2n-1} (-T_{2n,2n+1}) \langle 2n-1 \rangle$, with a   $(d(S^3_{-1}(T_{2n,2n+1})),d(S^3_{-1}(T_{2n,2n+1})))$ bigrading shift. Then  $C_n$  represents the local equivalence class of the complex $\CFK_{\F[\u,\v]}(-(Y_n,K_n)). $ See Figure \ref{fig:cn} for an example when $n=3.$
\begin{proposition} \label{prop:cn}
    For $n\geq 3,$ the complex  $C_n$ is characterized by
    \begin{align}
     \partial \alpha_s &= \begin{cases}
         \u^{\frac{n(n-1)}{2}}  \v^{\frac{n(n-1)}{2}}b^{(1)}_{n-1}, \qquad &s=1 \\
         \u^{\frac{n(n+1)}{2} -s + 1 } \v^{\frac{n(n+1)}{2}} b_{n}^{(s-1)} + \u^{\frac{n(n-1)}{2}}  \v^{\frac{n(n-1)}{2}}b^{(s)}_{n-1}, & 2\leq s \leq n-2\\
         \u^{\frac{n(n-1)}{2} + n -s + 1 } \v^{\frac{n(n+1)}{2}} b_{n}^{(s-1)} +  \u^{\frac{n(n+1)}{2}} \v^{\frac{n(n-1)}{2} - n + s + 1 }  b_{n}^{(s)} & n-1\leq s \leq n+1\\
          \u^{\frac{n(n-1)}{2}}  \v^{\frac{n(n-1)}{2}}b^{(s-1)}_{n+1} + \u^{\frac{n(n+1)}{2}}\v^{\frac{n(n-1)}{2} -n + s + 1 }  b_{n}^{(s)}, & n+2\leq s \leq 2n-2\\
           \u^{\frac{n(n-1)}{2}}  \v^{\frac{n(n-1)}{2}}b^{(2n-2)}_{n+1}, \qquad &s = 2n-1
           \end{cases}\\
     \partial \aa_s &= \begin{cases}     
     \u^n b_{n}^{(s)} + \v^{n-s-1} b_{n-1}^{(s)}, &\hspace{5em} 1\leq s \leq n-2\\
     \u^{s-n} b_{n+1}^{(s)} + \v^{n} b_{n}^{(s)},  &\hspace{5em} n+1\leq s \leq 2n-2.
     \end{cases}
  \end{align}
\end{proposition}
\begin{proof}
    This is a direct translation from Lemma \ref{le:cn}.
\end{proof}

This allows us to compute the values of the family of concordance homomorphisms $\varphi_{i,j}$ defined in \cite[Definition 8.1]{Homoconcor}, as follows.
\begin{proposition}
For each $n\geq 3,$ we have
    \begin{align}
      \varphi_{i,0}(C_n) &= \begin{cases}
              -1, \hspace{3em} &1\leq i \leq n-2\\
              -n+2, &i=n.
          \end{cases}\\
             \varphi_{\frac{n(n-1)}{2},\frac{n(n-1)}{2}}(C_n) &= -n+2, \\   
         \varphi_{\frac{n(n+1)}{2},j} (C_n) &= -1, \hspace{3em} \frac{n(n-1)}{2} \leq j \leq \frac{n(n+1)}{2} -1,       
    \end{align}
       and $\varphi_{i,j}(C_n)=0$ for all other $i$ and $j.$
\end{proposition}
\begin{proof}
    The complexes over $\F[\u,\v]$ can be translated to complexes over the ring $\mathbb{X}$ defined in \cite{Homoconcor} using the maps
\begin{align*}
    \u&\xmapsto[\hspace{2em}]{} U_B + W_{T,0}\\
    \v&\xmapsto[\hspace{2em}]{} V_T + W_{B,0}.
\end{align*}
Due to its simple form, it is not hard to formulate a change of a basis  under which $C_n$ becomes a standard complex (see \cite[Section 5.1]{Homoconcor}). In particular, the invariants $a_i$ of $C_n$ with $i$ odd (see \cite[Definition 6.1]{Homoconcor}) are given by the sequence
\begin{align*}
  \Bigg(  \underbrace{-\Big(\frac{n(n-1)}{2},\frac{n(n-1)}{2}\Big), -(n,0), } _{\text{repeats } n-2 \text{ times}} \cdots, -\Big(\frac{n(n+1)}{2}, \frac{n(n-1)}{2} \Big), -\Big(\frac{n(n+1)}{2}, \frac{n(n-1)}{2}+1 \Big),\\   
  \underbrace{-\Big(\frac{n(n+1)}{2}, \frac{n(n-1)}{2} +2 \Big), -(1,0), \cdots,  -\Big(\frac{n(n+1)}{2}, \frac{n(n-1)}{2} +s+1 \Big), -(s,0)}_{\text{for } 1\leq s\leq n-2 }, \cdots         \Bigg).  
\end{align*}
The computations for the values of $\varphi_{i,j}(C_n)$ immediately follow.
\end{proof}

Similarly, for the case $n=2$,   Lemma \ref{le:case12} yields the following. 
\begin{lemma} \label{le:phi12} We have
    \begin{align*}
        \varphi_{3,1}(C_2)  =  \varphi_{3,2}(C_2) = -1,
    \end{align*}
    and $\varphi_{i,j}(C_2)=0$ for all other $i$ and $j.$
\end{lemma}
As a consequence, we can compute the $\tau$ invariant of the manifold-knot pair $(Y_n, K_n)$.
\begin{proposition} \label{prop:taucn} For all $n\geq 1,$ 
   \[ \tau(Y_n, K_n) =  2n^2 - 3n +1. \]
\end{proposition}

\begin{proof}
    The $\tau$ invariant can be computed from $\varphi_{i,j}$ by \cite[Proposition 1.4]{Homoconcor}. For $n\geq 3,$ 
    \begin{align*}
        \tau(C_n) &= n(-n+2) - \sum^{n-2}_{i=1} i - \sum^{n}_{i=1} i \\
        &=- 2n^2 + 3n -1.
    \end{align*}
    When $n=2,$
    \[
    \tau(C_2) = -1 -2 = -3.
    \]
    The complex $C_1$ is  locally trivial, so  $\tau(C_1)=0$. The result now follows from the fact that $\tau$ is additive in the concordance group.
\end{proof}

According to \cite[Proposition 3.8]{OSknot}, the knot Floer complex of the mirror knot is the dual complex to the original knot. Therefore the local equivalence class of $\CFK(Y_n,K_n)$ is given by the dual complex of $C_n$; denote it by $C^*_n.$  Denote by  $\alpha_s^*$ and $\aa_s^*$ the dual of $\alpha_s,\aa_s$ respectively and similarly denote by  $b^{*,(s)}_i$ the dual of $b^{(s)}_i$.

 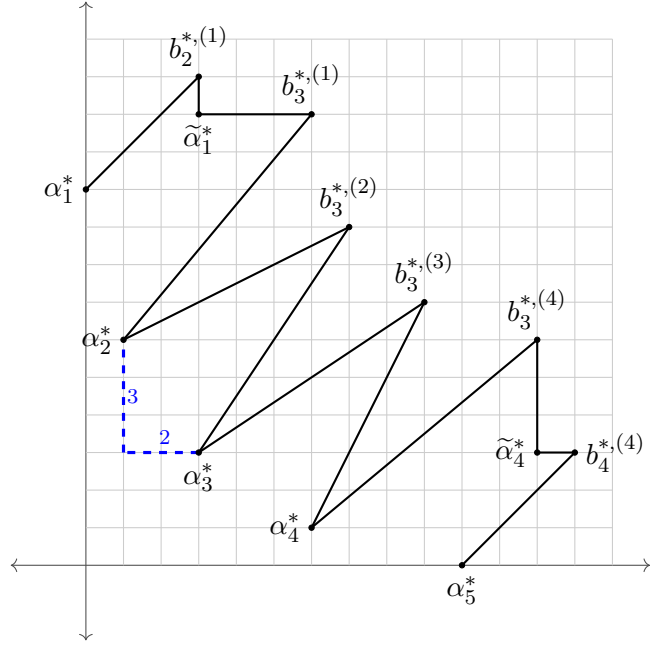
\begin{figure}[!htb]

\begin{tikzpicture}[scale=0.5]

\begin{scope}[thin, black!70!white]
		\draw [<->] (-2, 0) -- (15, 0);
	\draw [<->] (0, -2) -- (0, 15);
	\end{scope}
 
 \begin{scope}[thin, black!20!white]
   \foreach \i in {1,...,14}{
		\draw [-] (0, \i) -- (14, \i);
	\draw [-] (\i, 0) -- (\i, 14);
 }
	\end{scope}
     
\draw [blue, dashed, very thick] (3,3) -- (1, 3);
 \draw [blue, dashed, very thick] (1,3) -- (1, 6);

\node [blue] at (1.25,4.5) {\tiny  $3$};
\node [blue] at (2.1,3.4) {\tiny  $2$};
    
 \draw [thick] (0,10)--(3,13);
  \draw [thick] (3,12)--(3,13);
   \draw [thick] (3,12)--(6,12);
   \draw [thick] (1,6)--(6,12);
    \draw [thick] (1,6)--(7,9);
 \draw [thick] (3,3)--(7,9);
    \draw [thick] (3,3)--(9,7);
   \draw [thick] (6,1)--(9,7);
   \draw [thick] (6,1)--(12,6);
   \draw [thick] (12,3)--(12,6);
  \draw [thick] (12,3)--(13,3);
 \draw [thick] (10,0)--(13,3);

   \filldraw (0,10) circle (2pt) node () {};
			\node[left] at (0,10) {\small $\alpha_1^*$};
   \filldraw (3,12) circle (2pt) node () {};
			\node[below] at (3,12) {\small $\aa_1^*$};
    \filldraw (1,6) circle (2pt) node () {};
			\node[left] at (1,6) {\small $\alpha_2^*$};

    \filldraw (3,3) circle (2pt) node () {};
			\node[below] at (3,3) {\small $\alpha_3^*$};

       \filldraw (6,1) circle (2pt) node () {};
			\node[left] at (6,1) {\small $\alpha_4^*$};
          \filldraw (12,3) circle (2pt) node () {};
			\node[left] at (12,3) {\small $\aa_4^*$};

     \filldraw (10,0) circle (2pt) node () {};
			\node[below] at (10,0) {\small $\alpha_5^*$};

   \filldraw (3,13) circle (2pt) node () {};
			\node[above] at (3,13) {\small $b_{2}^{*,(1)}$};

     \filldraw (6,12) circle (2pt) node () {};
			\node[above] at (6,12) {\small $b_{3}^{*,(1)}$};

     \filldraw (7,9) circle (2pt) node () {};
			\node[above] at (7,9) {\small $b_{3}^{*,(2)}$};

    \filldraw (9,7) circle (2pt) node () {};
			\node[above] at (9,7) {\small $b_{3}^{*,(3)}$};

       \filldraw (12,6) circle (2pt) node () {};
			\node[above] at (12,6) {\small $b_{3}^{*,(4)}$};

       \filldraw (13,3) circle (2pt) node () {};
			\node[right] at (13,3) {\small $b_{4}^{*,(4)}$};
   
	\end{tikzpicture}

	\caption{The complex $C_3^*$, defined to be the dual complex of $C_3$. The axes indicate the $\u$ and $\v$ actions. The solid dots are generators, marked abstractly, missing actual $\u,\v$ decorations, and the edges  represent the differentials.}
	\label{fig:cndual}
	\end{figure}

\begin{proposition}\label{prop:cndual}
    For $n\geq 3,$ the complex $C^*_n$ is characterized by the following
    \begin{align}
        \partial b^{*,(s)}_{n-1} &=\u^{\frac{n(n-1)}{2}}\v^{\frac{n(n-1)}{2}}\alpha^*_s + \v^{n-s-1} \aa^*_s, \qquad  1\leq s \leq n-2\\
        \partial b^{*,(s)}_{n} &= \begin{cases}
            \u^{\frac{n(n+1)}{2}-s} \v^{\frac{n(n+1)}{2}} \alpha^*_{s+1} + \u^n \aa^*_s,   \qquad &1\leq  s \leq n-2\\
            \u^{\frac{n(n+1)}{2}}\v^{\frac{n(n-1)}{2}-n+s+1} \alpha_s^* + \u^{\frac{n(n+1)}{2}-s} \v^{\frac{n(n+1)}{2}} \alpha_{s+1}^*,    &n-1 \leq s \leq n\\
            \u^{\frac{n(n+1)}{2}}\v^{\frac{n(n-1)}{2}-n+s+1} \alpha_s^* + \v^n \aa_s^*,    &n+1 \leq s \leq 2n-2
        \end{cases}\\
        \partial b^{*,(s)}_{n+1} &=\u^{\frac{n(n-1)}{2}}\v^{\frac{n(n-1)}{2}}\alpha^*_{s+1} + \u^{s-n} \aa^*_s, \qquad  n+1\leq s \leq 2n-2
    \end{align}
\end{proposition}

\begin{proof}
    This follows from Proposition \ref{prop:cn} and the fact that $C^*_n$ is the dual complex of $C_n$.
\end{proof}

We record a few salient features of the complex $C^*_n$ for $n\geq 3$ from Proposition \ref{prop:cndual}:





\begin{lemma}\label{lem:gradings-of-cycles} 
We have the inequalities 
\[ \gr_\v \alpha_s^*,\gr_\v \aa_s^* \leq \gr_\v \alpha_n^*-2n, \quad \textup{ for } s\leq n-1.\]  Similarly, 
\[\gr_\u \alpha_s^*,\gr_\u \aa_s^* \leq \gr_\u \alpha_n^* -2n, \quad \textup{ for } s\geq n+1.\]
\end{lemma}
\begin{proof}
We have \[\gr_\v \alpha_{n-1}^*=\gr_\v \alpha_n^* -2n.\]
Note also the equalities:
\begin{align}\label{eq:v-grads-1}
\gr_\v \aa_s^*&=\gr_\v \alpha_{s+1}^* -n(n+1) \hspace{9.7em} \mbox{ for  } 1\leq s \leq n-2\\
\gr_\v \aa_s^* &=\gr_\v \alpha_{s}^*  -n(n-1) + 2(n-s-1) \qquad \qquad \mbox{ for } 1\leq s \leq n-2.
\label{eq:v-grads-2}\end{align}

In particular, 
\[
\gr_\v \alpha_{s}^* \leq \gr_\v \alpha_{s+1}^* -2n-2
\]
for $1\leq s\leq n-2.$

From here, the claim of the lemma follows for $\alpha_s^*$ for all $1\leq s \leq n-1$.  The statement for $\aa_s^*$ follows from (\ref{eq:v-grads-1}).

The case of $\gr_\u$ follows similarly, where we use 
\[
\gr_\u \alpha_{n+1}^*=\gr_\u \alpha_{n}^* -2n,
\]
and also calculate:
\begin{align}\label{eq:u-grads-1}
\gr_\u \aa_s^*&=\gr_\u \alpha_{s}^*-n(n+1) \hspace{9.7em} \mbox{ for  } n+1\leq s \leq 2n-2\\
\gr_\u \aa_s^*&=\gr_\u \alpha_{s+1}^*-n(n-1) + 2(s-n) \qquad \qquad \mbox{ for } n+1\leq s \leq 2n-2.
\label{eq:u-grads-2}\end{align}
In particular, 
\[
\gr_\u \alpha_{s+1}^* \leq \gr_\u \alpha_{s}^* -2n-2
\]
for $n+1\leq s\leq 2n-2.$
From here, the claim of the lemma follows for $\alpha_s^*$ for all $n+1\leq s \leq 2n-1$.  The statement for $\aa_s^*$ follows from (\ref{eq:u-grads-1}).
\end{proof}


\begin{lemma}\label{lem:constrained-destination}
For $n\geq 3$, let $\phi\colon C^*_n\to C^*_n$ be a chain map of bigrading $(c_1,c_2)$, where $c_1>-2n$ and $c_2>-2n$.  Then $\phi(\alpha_n^*)$ is either an $\F[\u, \v]$-multiple of $\alpha_n^*$ or $0$.
\end{lemma}
\begin{proof}
By Lemma \ref{lem:gradings-of-cycles}, all of the other generators of $C^*_n$ which are cycles have either $\u$-grading or $\v$-grading less than that of $\alpha_n^*$.  No linear combination of the $b^{*}$-type terms is a cycle, and so $\phi(\alpha_n^*)$ is supported only by $\langle \alpha_n^*\rangle$.  
\end{proof}


\begin{lemma}\label{lem:constrained-height}
For $n\geq 3$, let $\phi\colon C^*_n\to C^*_n$ be a homogeneous chain map with degree as in Lemma \ref{lem:constrained-destination} and so that $\phi(\alpha_n^*)$ is a boundary in $C^*_n\otimes \F[\u,\v=1]/(\u^{n-1})$.  Then $\phi(\alpha_n^*)$ must be divisible by $\u^{n-1}$.
\end{lemma}
\begin{proof}
From Lemma \ref{lem:constrained-destination}, $\phi(\alpha_n^*)=c\alpha_n^*$ for some $c\in \F[\u, \v]$.  Considering the differential of $C_n^*$ mod $\v=1,\u^{n-1}=0$, we obtain that $c\alpha_n^*$ is a boundary over this quotient ring if and only if $\u^{n-1}\mid c$. 
\end{proof}


We say that a chain complex $D$ over $\F[\u, \v]$ is \emph{$S^3$-knotlike} if $H_*(D\otimes \F[\u, \v=1]) = \F[\u]$. In particular, if $D$ is the knot Floer complex of a knot in $S^3$, then $D$ is $S^3$-knotlike.

\begin{lemma}\label{lem:surface-to-s3}
For $n\geq 3$, let $f$ be a map from $C_n^*$ to an $S^3$-knotlike complex $D$, and let $g$ be a map from $D$ to $C_n^*$.  Then $gf(\alpha_n^*)$ is a boundary in $C^*_n\otimes \F[\u,\v=1]/(\u^{n-1})$.
\end{lemma}
\begin{proof}
We have, by considering $b^{*,(n-1)}_n$, that 
\[
\u^{n(n-1)/2+1}\v^{n(n+1)/2} \alpha_n^* +\u^{n(n+1)/2}\v^{n(n-1)/2} \alpha_{n-1}^*
\]
is a boundary.  Setting $V=1$, 
\[
\u^{n(n-1)/2+1}(f( \alpha_n^*)+\u^{n(n+1)/2-n(n-1)/2-1}f( \alpha_{n-1}^*)) \mbox{ is a boundary in } C_n/(\v=1)
\]
Since any cycle in an $S^3$-knotlike complex that is $\u$-torsion in $(\v=1)$ homology is actually zero in homology, we have that 
\[
f(\alpha_n^*)+\u^{n-1}f(\alpha_{n-1}^*)
\] 
is a boundary in $D/(\v=1).$  So $f(\alpha_n^*)$ is a boundary in $D/(\v=1,\u^{n-1}=0)$.  Since $g$ is a chain map, the same holds for $gf(\alpha_n^*)$.  
\end{proof}



\begin{lemma}\label{lem:obstruction}
For $n\geq 3$, let $f$ be a local map from $C_n^*$ to a knotlike complex $D$.  There does not exist a local map $g\colon D\to C_n^*$, so that $g\circ f$ is of bigrading $(c_1,c_2)$ with $c_1>-2n+2$ and $c_2>-2n$.
\end{lemma}
\begin{proof}
    Combining the preceding lemmas, we obtain that 
    \[
    gf(\alpha_n^*)=\u^{-c_1/2}\v^{-c_2/2}\alpha_n^*.
    \]
    By Lemma \ref{lem:surface-to-s3}, $gf(\alpha_n^*)$ is a boundary mod $\u^{n-1}=0,\v=1$, and so $n-1\leq -c_1/2$.  That is, $-2n + 2 \geq c_1>-2n+2$, a contradiction.

\end{proof}




\emph{Proof of Theorem \ref{thm:cobdist}:} When $n=2,$   for any knot $J\subset S^3$, by \cite[Theorem 10.1]{Homoconcor} we have $\varphi_{i,j}(S^3,J)=0$ for any $ j\neq 0,$ so Lemma \ref{le:phi12} obstructs the existence of a homology concordance between $(Y_2,K_2)$ and $(S^3,J)$. 

Now suppose $n \geq 3.$  
Say that there is a pair $(W,S)$ as in the discussion preceeding Theorem \ref{thm:cobdist}, with $S$ of genus $g\leq n-2$.  Then, for any choice of $(c_1,c_2),(d_1,d_2)\in (2\mathbb{Z})^2$ so that $c_i,d_i\leq 0$ and $c_1+c_2=-2g=d_1+d_2$, there exist local maps $f\colon C_n\to \CFK(J)$ and $g\colon \CFK(J)\to C_n$ of bigrading $(c_1,c_2)$, $(d_1,d_2)$, respectively.  Let $f$ be of bigrading $(0,-2g)$ and $g$ be of bigrading $(-2g,0)$.  By hypothesis, $-2g\geq -2n+4$, and so Lemma \ref{lem:obstruction} applies to show that such $f,g$ do not exist, a contradiction.
\qed

\section{Preliminaries on the filtered mapping cone formula} \label{sec:mappingcone}
We start by reviewing the original definition of the knot Floer complex over the ring $\F[U,U^{-1}]$ by Ozsv\'{a}th and Szab\'{o}, as this is the setting where the filtered mapping cone formula can be  most conveniently defined.

In the original definition, the knot Floer complex is freely generated by the intersecting points of the two Lagrangians over the ring $\F[U,U^{-1}]$,  where the differentials similarly count the holomorphic disks, but are weighted over the intersection number with only  one of the basepoints. The data of the other basepoint is encoded in the Alexander grading. This version of the knot Floer complex is denoted by $\CFKi(Y,K)$, and commonly depicted in an $(i,j)$--plane, where the $j$--coordinate is given by the Alexander grading, and the $i$--coordinate is the normalized filtration level naturally induced by the $U$--action. We will often think of $\CFKi(Y,K)$ as a chain complex with an extra filtration given by the Alexander grading. By collapsing the Alexander filtration one recovers a chain complex associated to the underlying three-manifold, $\CFi(Y)$.

There is a Maslov grading on $\CFKi(Y,K)$, corresponding to $\gr_{\u}$; multiplication by $U$ on $\CFKi(Y,K)$ is equivalent to  multiplication by $\u\v$ on $\CFK_{\F[\u,\v]}(Y,K)$. Although the setting is slightly different, $\CFKi(Y,K)$ contains the same information as $\CFK_{\F[\u,\v]}(Y,K)$ does. In the setting of $\CFKi(Y,K)$, the local equivalence reads as follows.
\begin{definition}
    Two filtered chain complex $C_1$ and $C_2$ over $\F[U,U^{-1}]$ are \emph{locally equivalent} if there exist Maslov grading-preserving filtered local maps 
    \begin{align*}
        f \colon C_1 \rightarrow C_2  \hspace{2em} \text{and} \hspace{2em}     g \colon C_2 \rightarrow  C_1. 
    \end{align*}
\end{definition}

For the rest of the paper, we will always use the knot Floer complex $\CFKi(Y,K)$. Next, we recall  the filtered mapping cone formula from \cite{zhou2022filtered} for the reader; this is our main computational tool. 

Let $K\subset S^3$ be a knot with genus equal to $g.$ For a given positive integer $p$, let $\mu_{p,1}$ denote the $(p,1)$-cable of the meridian of $K$ in the $+1$-surgery on $K.$
According to \cite[Theorem 1.9]{zhou2022filtered}, the knot Floer complex  $\CFKi(S^3_{1}(K),\mu_{p,1})$ is filtered chain homotopy equivalent to the doubly-filtered chain complex $X^\infty_p(K)$, defined to be the mapping cone of
\begin{align} \label{eq:x_infinity}
     \bigoplus^{g+p-1}_{s=-g+1}A_s \xrightarrow{v_s+h_s} \bigoplus^{g+p-1}_{s=-g+2}B_s,
\end{align}
where each $A_s$ and $B_s$ are isomorphic to $\CFKi(S^3,K)$, coming with the $(i,j)$ coordinate. The map $v_s\co A_s \to B_s $ is the identity and the map $h_s\co A_s \to B_{s+1} $ is the reflection along $i=j$ precomposed with $U^s$. Note that there are corresponding versions of the filtered mapping cone formula for the hat, minus and infinity  flavors of knot Floer homology. In the following computation we will consistently use the infinity version of the $A_s$ and $B_s$  complexes and $v_s$ and $h_s$  maps, so we repress the superindices.

Let $\II$ and  $\JJ$  be the double filtrations 
 and let  $\gr_M$ be the absolute Maslov grading
on the filtered mapping cone complex $X^\infty_p(K)$. We will reserve letters $\II$ and  $\JJ$ solely for this purpose throughout the paper. We have
\begin{align}
\intertext{for $[\x,i,j] \in A_{s}$,}
\label{eq: filtration_s3_1}
 \II([\x,i,j]) &= \max\{i,j-s\} \\
 \label{eq: filtration_s3_2}
 \JJ([\x,i,j]) &= \max\{i-p,j-s\} + ps - \frac{p(p-1)}{2} \\
  \label{eq: grt-def-A} \gr_M([\x,i,j]) &= \widetilde{\gr}([\x,i,j]) + s(s-1)  \\
\intertext{and for $[\x,i,j] \in B_{s}$,}
 \label{eq: filtration_s3_3}
 \II([\x,i,j]) &= i \\
  \label{eq: filtration_s3_4}
 \JJ([\x,i,j]) &= i-p + ps - \frac{p(p-1) }{2}
  \\ \label{eq: grt-def-B} \gr_M([\x,i,j]) &= \widetilde{\gr}([\x,i,j]) + s(s-1)-1.
\end{align}
Here $\widetilde{\gr}$ denotes the absolute Maslov grading on the original chain complex $\CFKi(S^3,K)$.
 It is straightforward to check that for $s<-g+1$, the map $h_s$ induces an isomorphism on the homology; for $s>g+p-1,$ 
the map $v_s(K)$ induces an isomorphism on the homology,  which justifies the truncation of the mapping cone.

The general strategy for computation involves finding a \emph{reduced} basis for $X^\infty_p(K),$ where every term in the differential strictly lowers at least one of the filtrations. This can be achieved through a cancellation process (see for example \cite[Proposition 11.57]{Bordered}) as follows: suppose $\partial x_i = y_i$ $+ $ lower filtration terms, where the double filtration of $y_i$ is the same as $x_i$, then the subcomplex of  $X^\infty_p(K)$ generated by all such $\{x_i,  \partial x_i\}$ is acyclic, and $X^\infty_p(K)$ quotient by this complex is reduced. Alternatively,  one can view the above process as a change of basis, that splits off acyclic summands which individually lie entirely in one double-filtration level.  
 
There is an apparent symmetry on the mapping cone as follows. Let $[\x,i,j] \mapsto [\psi(\x),j,i]$ be a homotopy equivalence that realizes the symmetry on  the original chain complex $\CFKi(S^3,K)$. In the following lemma we use a subindex to mark elements from $A_s$ or $B_s$.
\begin{proposition} \label{prop:sym}
    Let $\Psi \co X^\infty_p(K) \longrightarrow X^\infty_p(K)$ be the map defined as
    \begin{align}
       \intertext{for $[\x,i,j]_s \in A_{s}$,}  \Psi([\x,i,j]_s) &= U^{\frac{(p-1)(p-2s)}{2}} [\psi(\x),j,i]_{p-s} \in A_{p-s}\\
       \intertext{for $[\x,i,j]_s \in B_{s}$,}  \Psi([\x,i,j]_{s}) &= U^{\frac{p(p-2s+1)}{2}} [\x,j,i]_{p-s+1} \in B_{p-s+1}.
    \end{align}
    Then $\Psi$ is a chain map that realizes a homotopy equivalence  on the doubly filtered chain complex $\CFKi(S^3_{1}(K),\mu_{p,1})$ that switches the $\II$ and $\JJ$ filtrations.
\end{proposition}
\begin{proof}
    By definition $\Psi$ is $U$--equivariant, so it suffices to show $\Psi$ realizes the symmetry for any one $\II$ and $\JJ$ value. 
    
    Over each chain complex $A_s,$ by \eqref{eq: filtration_s3_1} and \eqref{eq: filtration_s3_2}, we have $\{\II=0 \}=\max\{i,j-s\}$  and $\{\JJ=ps - \frac{p(p-1)}{2} \}=\max\{i-p,j-s\}$. Compute
    \begin{align*}
       \Psi(\{\II=0 \}_s) & = U^{\frac{(p-1)(p-2s)}{2}}\max\{i-s,j\}_{p-s}\\
       & = U^{\frac{(p-1)(p-2s)}{2}+(p-s)}\max\{i-p,j-(p-s)\}_{p-s}\\
       & = U^{\frac{p^2+p-2ps}{2}}\{\JJ=p(p-s) - \frac{p(p-1)}{2} \}_{p-s}\\
       & = \{\JJ=0\}_{p-s}  \\
      \Psi(\{\JJ=ps - \frac{p(p-1)}{2} \}_s) &= U^{\frac{(p-1)(p-2s)}{2}}\max\{i-s,j-p\}_{p-s} \\
      &=  U^{\frac{(p-1)(p-2s)}{2}-s}\max\{i,j-(p-s)\}_{p-s}\\
      &= \{\II=ps - \frac{p(p-1)}{2} \}_{p-s}
    \end{align*}
    while the computation for $B_s$  are similar and left for the reader. Moreover, by definition we have $\Psi\circ v_s =  h_s \circ \Psi,$ and $\Psi\circ h_s = v_s \circ \Psi,$ therefore $\Psi$ is a chain map.
\end{proof}

\section{Cables of the knot meridian of  $-T_{2n,2n+1}$}\label{sec:comp}
In this section, we perform the filtered mapping cone computation which determines the knot Floer complex in Proposition \ref{prop:cn} and Lemma \ref{le:phi12}.

Given $n\geq 1$, let $T_{2n,2n+1}$ be the $(2n,2n+1)$-torus knot, with genus equal to $n(2n-1).$ It is a fun exercise to compute its Alexander polynomial as follows
\begin{align*}
 \frac{(t^{2n(2n+1)}-1)(t-1)}{(t^{2n}-1)(t^{2n+1}-1)}=& \frac{t^{(2n-1)(2n+1)} + t^{(2n-2)(2n+1)} + \cdots + 1}{t^{2n-1} + t^{2n-2} + \cdots + 1}=\\
1+\sum^{2n-2}_{i=0}\big( t^{(2n-i)(2n-1)-i} &- t^{(2n-i)(2n-1)-2i-1}  \big).   
\end{align*}
For example, if we let $S_{2n}(i)= \big( t^{(2n-i)(2n-1)-i} - t^{(2n-i)(2n-1)-2i-1}  \big) \big( t^{2n-1} + t^{2n-2} + \cdots + 1 \big)$ for $i=0,1,\cdots, 2n-2$,  by induction we obtain that for $0 \leq \ell \leq 2n-2$
\[
\sum^{\ell}_{i=0} S_{2n}(i) = t^{(2n-1)(2n+1)} +  \cdots + t^{(2n-\ell -1)(2n+1)} -t^{(2n-\ell)(2n-1)-\ell -1 } - \cdots - t^{(2n-\ell)(2n-1)-2\ell -1}.
\]
Taking $\ell$ to be $2n-2$ leads to the answer.

Torus knots are $L$--space knots. Therefore according to \cite[Theorem 1.2]{OSlens}, the knot Floer complex $\CFKi(S^3,T_{2n,2n+1})$ is generated by $a^*_i$ with coordinate $(0, \frac{(2n-i)(2n-i+1)}{2} - \frac{i(i-1)}{2} )$ for $i \in \{1, \cdots, 2n\}$ and $b^*_i$ with coordinate $(0, \frac{(2n-i)(2n-i+1)}{2}- \frac{i(i+1)}{2} )$ for $i \in \{1, \cdots, 2n-1\}$ (this is in fact a set of generators coming from a $\widehat{HFK}$ model), where the differentials are given by
\[
 \partial b^*_i = U^i a^*_i + a^*_{i+1}.
\]
It follows from \cite[Proposition 3.8]{OSknot}, that the knot Floer complex of the mirror knot is the dual complex to the original knot. Therefore $\CFKi(S^3,-T_{2n,2n+1})$ is generated by $a_i$ with coordinate $(0, -\frac{(2n-i)(2n-i+1)}{2} + \frac{i(i-1)}{2} )$ for $i \in \{1, \cdots, 2n\}$ and $b_i$ with coordinate $(0, -\frac{(2n-i)(2n-i+1)}{2}+\frac{i(i+1)}{2} )$ for $i \in \{1, \cdots, 2n-1\}$ (simply by taking $a_i$ to be the dual of $a^*_i$ and $b_i$ to be the dual of $b^*_i$). As a notational shorthand, we will let $g_a(n,i) \coloneqq -\frac{(2n-i)(2n-i+1)}{2} + \frac{i(i-1)}{2}$ and $g_b(n,i) \coloneqq -\frac{(2n-i)(2n-i+1)}{2} + \frac{i(i+1)}{2}.$  Note that $g_a(n,i) +i = g_b(n,i)$. The differentials are given by
\begin{align*}
     \partial a_i = 
     \begin{cases} Ub_{i}, &\qquad i=1\\
     b_{i-1}, &\qquad i=2n\\
       U^i b_{i} +   b_{i-1}, &\qquad \text{otherwise.}
     \end{cases}
\end{align*}
Note that the (horizontal) arrow from $a_i$ to $b_i$ is of length $i$ while the (vertical) arrow from $a_i$ to $b_{i-1}$ is of length $2n-i+1$.  See Figure \ref{fig:mT67} for an example of  $\CFKi(S^3,-T_{2n,2n+1})$ when $n=3.$

 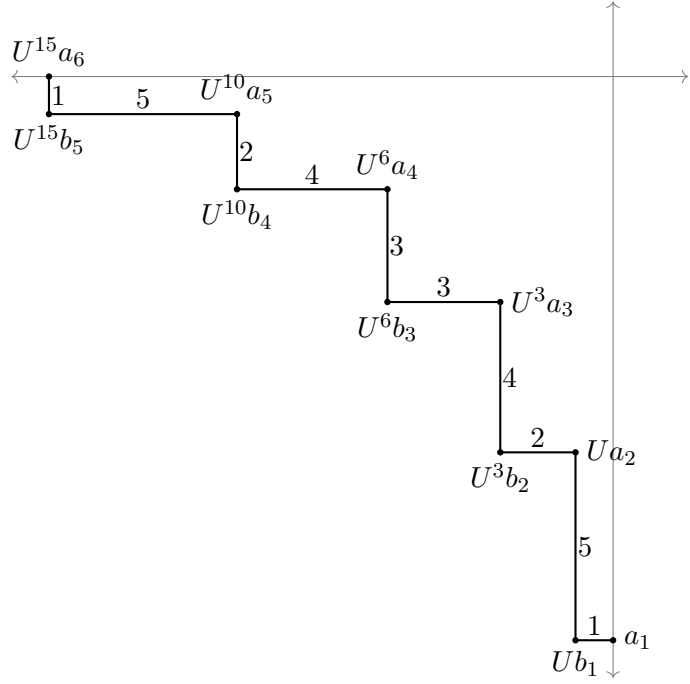
\begin{figure}[!htb]

\begin{tikzpicture}[scale=0.5]

\begin{scope}[thin, black!50!white]
		\draw [<->] (-16, 0) -- (2, 0);
	\draw [<->] (0, -16) -- (0, 2);
	\end{scope}
		\filldraw (-15,0) circle (2pt) node (a_{6}) {};
			\node[above] at (a_{6}) {\small $U^{15} a_{6}$};
       \node at (-14.75,-0.5) {\small $1$};
       
   \filldraw (-15,-1) circle (2pt) node (b_{5}) {};
			\node[below] at (b_{5}) {\small $U^{15} b_{5}$};
    \draw [thick] (-15,0)--(-15,-1);
\draw [thick] (-10,-1)--(-15,-1);
  \node at (-12.5,-0.6) {\small $5$};
  
    	\filldraw (-10,-1) circle (2pt) node (a_{5}) {};
			\node[above] at (a_{5}) {\small $U^{10} a_{5}$};

   \filldraw (-10,-3) circle (2pt) node (b_{4}) {};
			\node[below] at (b_{4}) {\small $U^{10} b_{4}$};
    \draw [thick] (-10,-1)--(-10,-3);
    \node at (-9.75,-2) {\small $2$};

\draw [thick] (-10,-3)--(-6,-3);
\node at (-8,-2.6) {\small $4$};

    	\filldraw (-6,-3) circle (2pt) node (a_{4}) {};
			\node[above] at (a_{4}) {\small $U^6 a_{4}$};

   \filldraw (-6,-6) circle (2pt) node (b_{3}) {};
			\node[below] at (b_{3}) {\small $U^6 b_{3}$};
    \draw [thick] (-6,-3)--(-6,-6);
    \node at (-5.75,-4.5) {\small $3$};

    \draw [thick] (-6,-6)--(-3,-6);
\node at (-4.5,-5.6) {\small $3$};
    	\filldraw (-3,-6) circle (2pt) node (a_{3}) {};
			\node[right] at (a_{3}) {\small $U^3 a_{3}$};

   \filldraw (-3,-10) circle (2pt) node (b_{2}) {};
			\node[below] at (b_{2}) {\small $U^3 b_{2}$};
    \draw [thick] (-3,-6)--(-3,-10);
     \node at (-2.75,-8) {\small $4$};
     
    \draw [thick] (-3,-10)--(-1,-10);
    \node at (-2,-9.6) {\small $2$};

    	\filldraw (-1,-10) circle (2pt) node (a_{2}) {};
			\node[right] at (a_{2}) {\small $Ua_{2}$};

   \filldraw (-1,-15) circle (2pt) node (b_{1}) {};
			\node[below] at (b_{1}) {\small $Ub_{1}$};
    \draw [thick] (-1,-10)--(-1,-15);
     \node at (-0.75,-12.5) {\small $5$};
    \draw [thick] (-1,-15)--(0,-15);
    \node at (-0.5,-14.6) {\small $1$};
    	\filldraw (0,-15) circle (2pt) node (a_{1}) {};
			\node[right] at (a_{1}) {\small $a_{1}$};

	\end{tikzpicture}

	\caption{The knot Floer complex $\CFKi(S^3,-T_{6,7})$.  The solid dots are generators. The  differentials point to lower  filtration levels, and the numbers indicate their lengths.}
	\label{fig:mT67}
	\end{figure}

The interesting examples are given by the pair $(S^{3}_{1}(-T_{2n,2n+1}),\mu_{2n-1,1})$, where $\mu_{2n-1,1}$ is the $(2n-1,1)$-cable of the dual knot. To compute the knot Floer complex of said examples, we apply the filtered mapping cone formula for the cables of the dual knot on $-T_{2n.2n+1}$, with the surgery coefficient equal to $+1.$  Following the recipe described in Section \ref{sec:mappingcone},  the filtered chain complex $\CFKi(S^{3}_{1}(-T_{2n,2n+1}),\mu_{2n-1,1})$ is filtered  homotopy equivalent to the filtered complex  $X^\infty_{2n-1}(-T_{2n,2n+1})$ defined by the mapping cone of
 
\[   \bigoplus^{(n+1)(2n-1)-1}_{s=-n(2n-1)+1}A_s \xrightarrow{v_s+h_s} \bigoplus^{(n+1)(2n-1)-1}_{s=-n(2n-1)+2}B_s.\]

Through the isomorphism with $\CFKi(S^3,-T_{2n,2n+1})$,  denote the corresponding generators in $A_s$ by $a^{(s)}_i$ and $b^{(s)}_i$, and  the generators in $B_s$ by $a'^{(s)}_i$ and $b'^{(s)}_i$, for suitable $i$ and $s$. Recall that we use $\II$ and $\JJ$ specifically for the double filtrations on the entire mapping cone complex.  Using the formulas given by \eqref{eq: filtration_s3_1}, \eqref{eq: filtration_s3_2}, \eqref{eq: filtration_s3_3} and \eqref{eq: filtration_s3_4}, the computations for the $\II$ and $\JJ$ filtrations of the  generators described above are quite straightforward.  We collect the result in a following lemma, with $g_a(n,i)$ and $g_b(n,i)$ the quantities defined in the previous paragraph.
Also define a notational shorthand 
 \begin{equation}\label{eq: def_f}
      f(n,s) \coloneqq -\frac{(n-1)n}{2} + ns.
 \end{equation}
Note that $f(n,s-1) + n = f(n,s).$   
\begin{lemma}\label{le:iijj}
In the complex $X^\infty_{2n-1} (-T_{2n,2n+1})$, we have
\begin{align}
 \label{eq:ai} \JJ(a^{(s)}_i) &= \begin{cases}f(n,s)+g_a(n,i)-s,  \qquad   &s \leq  g_a(n,i)+2n-1 \\
    f(n,s-1),  \qquad   &s >  g_a(n,i)+2n-1
        \end{cases}\\
 \label{eq:bi} \JJ(b^{(s)}_i) &= \begin{cases}f(n,s)+g_b(n,i)-s,  \qquad   &s \leq  g_b(n,i)+2n-1 \\
    f(n,s-1),  \qquad   &s >  g_b(n,i)+2n-1;
        \end{cases}\\
         \II(a^{(s)}_i) &= \II(b^{(s)}_i) =   \II(a'^{(s)}_i) = \II(b'^{(s)}_i) = 0;\\
    \JJ(a'^{(s)}_i) &= \JJ(b'^{(s)}_i) = f(n,s-1).
\end{align}  
\end{lemma}

For the rest of the computation, we assume that $n \geq 3.$ (The case when $n=1, 2$ does not fit into the following model. Instead, the results of those two cases are recorded in Lemma \ref{le:case12}.)

We first aim to obtain a reduced model of the generators of $X^\infty_{2n-1} (-T_{2n,2n+1})$.  Up to filtered homotopy equivalence, (as a subcomplex of $X^\infty_{2n-1} (-T_{2n,2n+1})$) each  $B_s$ is one-dimensional. Indeed, quotienting out $\{a'^{(s)}_i,\partial a'^{(s)}_i\}_{2\leq i\leq 2n}$ leaves us with a sole generator $ a'^{(s)}_1$ in each $B_s$.

Each $A_s$ is a subcomplex of the quotient complex $\bigoplus_s A_s$, which inherits the $(\II,\JJ)$ filtration naturally. We would like to obtain a reduced model for each $A_s$.  For the next part, let  $\partial$ temporarily denote the differential restricted to each sub-quotient-complex $A_s$, as opposed to the differential on the entire chain complex $X^\infty_{2n-1} (-T_{2n,2n+1})$.      There are two types of complex $A_s$, depending on the dimension of the reduced model. 

When $s\in [-n(2n-1),-n(2n-1)+2n] \cup \{-n(2n-1) + 2jn-1, -n(2n-1) + 2jn\}_{1\leq j \leq 2n-2} \cup [n(2n-1)-1,(n+1)(2n-1) -1] ,$ after quotienting out $\{a^{(s)}_i, \partial a^{(s)}_i\}_{2\leq i\leq 2n-1}$, the reduced model of $A_s$ is $3$-dimensional. 
\begin{itemize}
    \item When $s\in [-n(2n-1),-n(2n-1)+2n]$,  the reduced model is generated by $\{a^{(s)}_{2n},b^{(s)}_{1},a^{(s)}_{1}\}$  with modified differentials:
    \[
\partial a^{(s)}_{1}= U b^{(s)}_{1}, \hspace{6em}  \partial a^{(s)}_{2n}= U^{-n(2n-1)+1} b^{(s)}_{1}.
    \]
    \item When $s\in \{-n(2n-1) + 2jn-1, -n(2n-1) + 2jn\}_{1\leq j \leq 2n-2}$,the reduced model is generated by $\{a^{(s)}_{2n},b^{(s)}_{j},a^{(s)}_{1}\}$ and modified differentials are 
      \[
\hspace{1.2em}  \partial  a^{(s)}_{1}= U^{\frac{j(j+1)}{2}} b^{(s)}_{j}, \hspace{4em}  \partial a^{(s)}_{2n}= U^{-n(2n-1)+\frac{j(j+1)}{2}} b^{(s)}_{j}.
    \]
    \item When $s\in [n(2n-1)-1,(n+1)(2n-1) -1]$, the reduced model is generated by $\{a^{(s)}_{2n},b^{(s)}_{2n-1},a^{(s)}_{1}\}$ with modified differentials 
      \[
\partial a^{(s)}_{1}= U^{n(2n-1)} b^{(s)}_{2n-1}, \hspace{5em}  \partial a^{(s)}_{2n}=  b^{(s)}_{2n-1}. \hspace{2em} 
    \]
\end{itemize}
When $s\in  \bigcup_{1\leq j \leq 2n-2} [-n(2n-1) + 2jn+1, -n(2n-1) + 2(j+1)n-2]$, the reduced model of $A_s$ is $5$-dimensional. Indeed, quotienting out $\{a^{(s)}_i, \partial a^{(s)}_i\}_{j\in[2,j]\cup [j+2,2n-1]}$ leaves us with generators  $\{a^{(s)}_{2n}, b^{(s)}_{j+1}, a^{(s)}_{j+1}, b^{(s)}_{j},a^{(s)}_{1}\}$. The difference here from the previous case is that both terms in $\partial a^{(s)}_{j+1}$ strictly decrease $\II$ or $\JJ$ grading, and therefore survive into the reduced complex. The modified differentials are given by
  \begin{gather*}
  \begin{aligned}
      &\partial a^{(s)}_{1}= U^{\frac{j(j+1)}{2}} b^{(s)}_{j}, \hspace{3em}  &\partial a^{(s)}_{j+1}=  b^{(s)}_{j} + U^{j+1}  b^{(s)}_{j+1},\\
      & \partial a^{(s)}_{2n} = U^{-n(2n-1)+\frac{(j+1)(j+2)}{2}} b^{(s)}_{j+1}.         &
  \end{aligned}
  \end{gather*}
Finally, consider  the entire chain complex $X^\infty_{2n-1} (-T_{2n,2n+1})$, using the reduced models for both $A_s$ and $B_s.$ Let $\partial$ denote the differential on the entire mapping cone complex (including $v_s$ and $h_s$ maps). Observe that $h_s(U^{-s} a^{(s)}_{2n})=a'^{(s+1)}_{1}= v_{s+1}(a^{(s+1)}_{1})$ for $-n(2n-1)+1\leq s \leq n(2n-1)$, while $\II(U^{-s} a^{(s)}_{2n})=\II(a'^{(s+1)}_{1}) = \II(a^{(s+1)}_{1})$ and $\JJ(U^{-s} a^{(s)}_{2n})=\JJ(a'^{(s+1)}_{1})\leq \JJ(a^{(s+1)}_{1})$, where the last equality is reached when $s\geq -(n-1)(2n-1).$ Thus we may quotient out $\{a^{(s)}_{2n}, \partial a^{(s)}_{2n}\}$ for $-n(2n-1)+1\leq s \leq n(2n-1)$. If we let $\alpha_s$ denote the image of $a^{(s)}_{1}$ in the quotient for $-n(2n-1)+1 \leq s \leq n(2n-1)+1$, notice that for $-n(2n-1)+2 \leq s \leq n(2n-1)+1$  this amounts to a change of basis $ a^{(s)}_{1} \mapsto U^{-s+1}a^{(s-1)}_{2n} + a^{(s)}_{1}$ followed by a homotopy equivalence. Similarly, we may quotient out  $\{a^{(s)}_{1}, \partial a^{(s)}_{1}\}$ for $n(2n-1)+2\leq s \leq (n+1)(2n-1)-1$.

We have obtained a reduced model for $X^\infty_{2n-1} (-T_{2n,2n+1})$. Observe that no generator in $B_s$ survives into the reduced basis. Moreover, from the viewpoint of the quotient complex,   the induced differential $\partial$ restricted to $A_s$ is a map $\partial\co A_s \rightarrow A_s \oplus A_{s-1}$ for $-n(2n-1)+2\leq s\leq n$, viewing $\alpha_s$ as an element of $A_s$. However, we will generally adopt the viewpoint of a change of basis, and view  $\alpha_s$ as an element of $A_s \oplus A_{s-1}$, mainly because this plays well with the symmetry on the mapping cone complex.


Considering the symmetry on $X^\infty_{2n-1} (-T_{2n,2n+1})$  (see Proposition \ref{prop:sym}), our strategy would be to focus on the ``first half'' of the complex, namely the mapping cone of
\begin{align*}
    \bigoplus_{s=-n(2n-1)+1}^{ n-1} A_s  \xrightarrow{ v_s+h_s} \bigoplus_{s=-n(2n-1)+2}^{n} B_s,
\end{align*}
which under the current basis is simply the chain complex 
\begin{align*}
   \bigoplus_{s=-n(2n-1)+1}^{n-1} A_s. 
\end{align*}
So let us summarize the generators and relations of this first half complex in the following lemma. (We also include those $A_s$ where $s$ is in the interval $[n, 2n-1]$ for the continuity.) Let $\aa_s$ denote $a^{(s)}_{j+1}$ when $s\in  [-n(2n-1) + 2jn+1, -n(2n-1) + 2(j+1)n-2]$ for each $1\leq j \leq n.$ 
\begin{lemma} \label{le:half}
Under the reduced basis chosen above, we have
\begin{itemize}
    \item For $s\in [-n(2n-1)+1, -n(2n-1)+2n],$  the complex $A_s$ is generated by $\alpha_s$ and $ b^{(s)}_{1}$, where the differentials are given by
    \begin{align}
     \partial \alpha_s &= \label{eq:dif1}
     \begin{cases} Ub^{(s)}_{1}, &\qquad s=-n(2n-1)+1  \hspace{8em}  \\
       U^{-s+2}b^{(s-1)}_{1} + Ub^{(s)}_{1} ,  &\qquad s>-n(2n-1)+1.
     \end{cases}
\end{align}
\item For $s\in [-n(2n-1) + 2jn+1, -n(2n-1) + 2(j+1)n-2]$ with some $1\leq j \leq n,$  the complex $A_s$ is generated by $\alpha_s, \aa_s, b^{(s)}_{j}$ and $ b^{(s)}_{j+1}$, where the differentials are given by
    \begin{align}
    \label{eq:dif2} \partial \alpha_s &= \begin{cases}   U^{-s+1+\frac{j(j+1)}{2}} b^{(s-1)}_{j} + U^{\frac{j(j+1)}{2}} b^{(s)}_{j} ,  &\qquad s=-n(2n-1) + 2jn+1,\\
         U^{-s+1+\frac{(j+1)(j+2)}{2}} b^{(s-1)}_{j+1} + U^{\frac{j(j+1)}{2}} b^{(s)}_{j} ,  &\qquad s>-n(2n-1) + 2jn+1,
     \end{cases}\\
    \label{eq:dif3} \partial \aa_s &= b^{(s)}_{j} + U^{j+1}  b^{(s)}_{j+1}.
\end{align}
\item For $s\in \{-n(2n-1) + 2jn-1, -n(2n-1) + 2jn\}$ with some $2\leq j \leq n,$
the complex $A_s$ is generated by $\alpha_s$ and $ b^{(s)}_{j}$, where the differentials are given by
\begin{align}
   \label{eq:dif4} &\partial \alpha_s = U^{-s+1+\frac{j(j+1)}{2}} b^{(s-1)}_{j} + U^{\frac{j(j+1)}{2}} b^{(s)}_{j}. \hspace{16.5em}
\end{align}
\end{itemize}
\end{lemma}
\begin{proof}
    This follows from the earlier discussion.
\end{proof}

 We prove in the next lemma that up to local equivalence, we can further truncate the mapping cone.
Define  $X^\infty_{2n-1} (-T_{2n,2n+1}) \langle \l \rangle $ for $\l\in \Z$ to be the filtered mapping cone
 \begin{align*}
    \bigoplus^{\l}_{s=-\l+2n-1}A_s \xrightarrow{v_s+h_s} \bigoplus^{\l}_{s=-\l+2n}B_s,
\end{align*}
which  under the reduced basis   simplifies to the filtered chain complex 
 \begin{align*}
    \bigoplus^{\l}_{s=-\l+2n-1}A_s.
\end{align*}
Note that under this notation  $X^\infty_{2n-1} (-T_{2n,2n+1}) = X^\infty_{2n-1} (-T_{2n,2n+1}) \langle (n+1)(2n-1) -1 \rangle$. 
 \begin{lemma}\label{le: localequi}
Up to a change of basis, the filtered complex $X^\infty_{2n-1} (-T_{2n,2n+1})$  is isomorphic to $ X^\infty_{2n-1} (-T_{2n,2n+1}) \langle 2n-1 \rangle \oplus D $, where $H_*(D)=0$.
 \end{lemma}
\begin{proof}
It suffices to show for any $2n \leq \l \leq (n+1)(2n-1) -1,$ the complex $X^\infty_{2n-1} (-T_{2n,2n+1}) \langle \l \rangle$ is isomorphic to $X^\infty_{2n-1} (-T_{2n,2n+1}) \langle \l -1 \rangle \oplus D'$ up to a change of basis, where $H_*(D')=0$. For every such $\l$, we will demonstrate a filtered change of basis such that the complex $A_{-\l+2n-1}$
 becomes a summand.  Following from the symmetry given by Proposition \ref{prop:sym}, there is also a filtered change of basis such that $A_{\l}$ 
 becomes a summand under the new basis as required. Let $s=-\l+2n-1.$ Recall that we view $\alpha_s$ as an element in $A_{s}\oplus A_{s-1}.$ 
 \begin{itemize}
     \item For $s\in [-n(2n-1)+2, -n(2n-1)+2n+1],$ perform the change of basis
     \begin{align*}
         \alpha_{s} \xmapsto[\qquad\null]{}  &\alpha_{s} + U^{-s+1} \alpha_{s-1}.  
     \end{align*}
     According to \eqref{eq:dif1} and \eqref{eq:dif2}, this splits off an acyclic summand as required.
     Since $\JJ(\alpha_{s}) = \JJ(a_1^{s})>\JJ(a_1^{s-1}) = \JJ(\alpha_{s-1})$ by \eqref{eq:ai}, this change of basis is clearly filtered.
     \item For $s\in [-n(2n-1) + 2jn+2, -n(2n-1) + 2(j+1)n-1]$ for some $1\leq j \leq n-1,$ and when $s\leq 1$,
     perform the change of basis
     \begin{align*}
        \hspace{6em} \alpha_{s} \xmapsto[\qquad\null]{}  &\alpha_{s}+U^{-s+1}\big(\alpha_{s-1} + U^{\frac{j(j+1)}{2}} \aa_{s-1}\big).  
     \end{align*}
     According to \eqref{eq:dif2}, \eqref{eq:dif3}  and \eqref{eq:dif4}, this splits off an acyclic summand as required.
     This change of basis is clearly filtered when $s\leq 1.$ (The equality is reached in the interval associated to $j=n-1$.)
     \item For $s\in \{-n(2n-1) + 2jn, -n(2n-1) + 2jn+1\}$ with some $2\leq j \leq n-1$ (noting that $s<0$ always holds), perform the change of basis
         \begin{align*}
         \alpha_{s} \xmapsto[\qquad\null]{}  &\alpha_{s} + U^{-s+1} \alpha_{s-1}.  
     \end{align*}
     This change of basis is again clearly filtered.
 \end{itemize}
\end{proof}
Therefore the local equivalence class of $X^\infty_{2n-1} (-T_{2n,2n+1})$ is given by $\bigoplus^{2n-1}_{s=0}A_s$ under the reduced basis. The differentials in this complex are already  given by Lemma \ref{le:half} and the filtrations of the generators are given by  Lemma \ref{le:iijj}. In the following lemma we will work out the $\JJ$--filtration shifts between the generators that are related by a differential.

Suppose $U^c \beta$ is a nontrivial term in $\partial \alpha$, where $\beta$ is used to represent some $b_i^{(s)}$ and $\alpha$ is used to represent some $\alpha_s$ or $\aa_s$. Define 
\begin{equation}
    \DD(\alpha,\beta) = (\II,\JJ)(\alpha) - (\II,\JJ)(U^c \beta)
\end{equation}
and similarly define $\Delta_\II$ and $\Delta_\JJ$.

\begin{lemma} \label{le:cn}
    Generators in the reduced basis of $\bigoplus^{2n-1}_{s=0}A_s$ satisfy the following.
    \begin{align}
       \label{eq:DD1} \DD(\alpha_s,b_{n-1}^{(s)}) &= \big( \frac{n(n-1)}{2}, \frac{n(n-1)}{2} \big),   && 1\leq s \leq n-2 \\
        \label{eq:DD2} \DD(\alpha_s,b_{n+1}^{(s-1)}) &= \big( \frac{n(n-1)}{2}, \frac{n(n-1)}{2} \big),   && n+2\leq s \leq 2n-1 \\
        \label{eq:DD3} \DD(\alpha_s,b_{n}^{(s-1)}) &= \big( \frac{n(n+1)}{2} -s + 1 , \frac{n(n+1)}{2} \big),   && 2\leq s \leq n \\
        \label{eq:DD4} \DD(\alpha_s,b_{n}^{(s)}) &= \big( \frac{n(n+1)}{2}  , \frac{n(n-1)}{2} +s - n + 1 \big),   && n\leq s \leq 2n-2 \\
         \label{eq:DD5} \DD(\alpha_{n-1},b_{n}^{(n-1)}) &= \big( \frac{n(n+1)}{2}  , \frac{n(n-1)}{2} \big),   &&  \\
        \label{eq:DD6} \DD(\alpha_{n+1} ,b_{n}^{(n)}) &= \big( \frac{n(n-1)}{2}  , \frac{n(n+1)}{2}  \big),   && \\
         \label{eq:DD7} \DD(\aa_s,b_{n-1}^{(s)}) &= \big( 0 , n-1-s \big), \hspace{1.5em}    \DD(\aa_s,b_{n}^{(s)}) = \big( n , 0 \big),  \hspace{0.5em}  && 1\leq s \leq n-2 \\
         \label{eq:DD8} \DD(\aa_s,b_{n}^{(s)}) &= \big( 0 , n \big),     \hspace{2.8em}     \DD(\aa_s,b_{n+1}^{(s)}) = \big( s-n , 0 \big),  &&  n+1\leq s \leq 2n-2.
    \end{align}
\end{lemma}

\begin{proof}
    We collect in  Table \ref{ta:gen}  the filtrations of the generators in the reduced basis of $\bigoplus^{2n-1}_{s=0}A_s$ from Lemma \ref{le:iijj}. Note that $g_b(n,n)=0$ and $g_a(n,n)=-n$. The $\II$ filtrations of the generators are all $0$ (so this is in fact a reduced model of $\widehat{HFK}.$)
    
   To show \eqref{eq:DD1} and \eqref{eq:DD2}, first by \eqref{eq:dif2} we have $\Delta_{\II}(\alpha_s,b_{n-1}^{(s)}) = \frac{n(n-1)}{2}.$ Compute
  \begin{align*}
      \Delta_{\JJ}(\alpha_s,b_{n-1}^{(s)}) &= \JJ(\alpha_s) - \JJ(b_{n-1}^{(s)}) + \frac{n(n-1)}{2}\\
      &= \frac{n(n-1)}{2},
  \end{align*}
  which proves \eqref{eq:DD1}, and \eqref{eq:DD2} follows from the symmetry given by Proposition \ref{prop:sym}.

  To show \eqref{eq:DD3} and \eqref{eq:DD4}, first by \eqref{eq:dif2} and \eqref{eq:dif4} we have $\Delta_{\II}(\alpha_s,b_{n}^{(s-1)}) = \frac{n(n+1)}{2} - s + 1.$ Compute 
  \begin{align*}
      \Delta_{\JJ}(\alpha_s,b_{n}^{(s-1)}) &= \JJ(\alpha_s) - \JJ(b_{n}^{(s-1)}) + \frac{n(n+1)}{2} - s + 1\\
      &= \frac{n(n+1)}{2},
  \end{align*}
  which proves \eqref{eq:DD3}, and \eqref{eq:DD4} follows from the symmetry given by Proposition \ref{prop:sym}.

  The rest of the results follow from similar computations and are left for the reader.
   \begin{table}
   \centering
   \begin{tabular}{|c|c|c|}\hline
  Generators  & The $\JJ$--filtrations & Range \\\hline
  $\alpha_s$ & $f(n,s-1)$ & $1\leq s \leq 2n-1$ \\\hline
\multirow{2}{*}{$\aa_s$} & $f(n,s-1)+n-1-s$ & $1\leq s \leq n-2$ \\\cline{2-3}
    & $f(n,s-1)$ & $n+1 \leq s \leq 2n-2$ \\
  \hline
   $b_n^{(s)}$ & $f(n,s)-s$ & $1\leq s \leq 2n-2$ \\\hline
   $b_{n-1}^{(s)}$ & $f(n,s-1)$ & $1 \leq s \leq n-2$ \\\hline
   $b_{n+1}^{(s)}$ & $f(n,s)+2n+1-s$ & $n+1\leq s \leq 2n-2$ \\\hline
\end{tabular}\caption{The filtrations of the generators in the reduced basis of $\bigoplus^{2n-1}_{s=0}A_s$.}\label{ta:gen}
 \end{table}

\end{proof}


 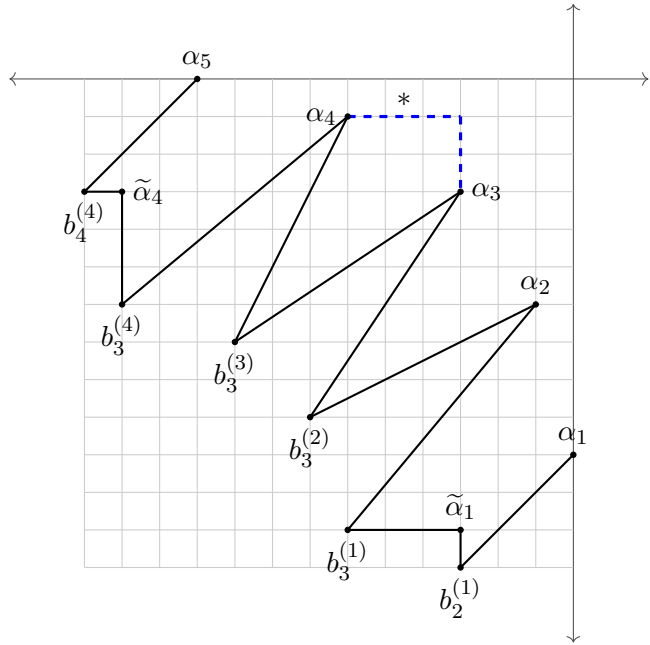
\begin{figure}[!htb]

\begin{tikzpicture}[scale=0.5]

\begin{scope}[thin, black!70!white]
		\draw [<->] (-2, 13) -- (15, 13);
	\draw [<->] (13, -2) -- (13, 15);
	\end{scope}
 
 \begin{scope}[thin, black!20!white]
   \foreach \i in {0,...,12}{
		\draw [-] (0, \i) -- (13, \i);
	\draw [-] (\i, 0) -- (\i, 13);
 }
	\end{scope}
		\filldraw (13,3) circle (2pt) node () {};
			\node[above] at (13,3) {\small  $\alpha_{1}$};
     
\draw [blue, dashed, very thick] (10,12) -- (10, 10);
 \draw [blue, dashed, very thick] (7,12) -- (10, 12);
 
\node[above] at (8.5,12) {\small $ *$};

   \filldraw (10,0) circle (2pt) node () {};
			\node[below] at (10,0) {\small $ b_{2}^{(1)}$};
    \draw [thick] (13,3)--(10,0);
\draw [thick] (10,0)--(10,1);
   \filldraw (10,1) circle (2pt) node () {};
			\node[above] at (10,1) {\small $ \aa_1$};

\draw [thick] (7,1)--(10,1);
    \filldraw (7,1) circle (2pt) node () {};
			\node[below] at (7,1) {\small $ b_{3}^{(1)}$};

   \filldraw (12,7) circle (2pt) node () {};
			\node[above] at (12,7) {\small  $\alpha_{2}$};
     \draw [thick] (12,7)--(7,1);

\filldraw (6,4) circle (2pt) node () {};
			\node[below] at (6,4) {\small  $b_{3}^{(2)}$};
     \draw [thick] (12,7)--(6,4);

        \filldraw (10,10) circle (2pt) node () {};
			\node[right] at (10,10) {\small  $\alpha_{3}$};
     \draw [thick] (10,10)--(6,4);

\filldraw (4,6) circle (2pt) node () {};
			\node[below] at (4,6) {\small  $b_{3}^{(3)}$};
     \draw [thick] (10,10)--(4,6);

      \filldraw (7,12) circle (2pt) node () {};
			\node[left] at (7,12) {\small  $\alpha_{4}$};
     \draw [thick] (7,12)--(4,6);

   \filldraw (1,7) circle (2pt) node () {};
			\node[below] at (1,7) {\small $ b_{3}^{(4)}$};
    \draw [thick] (1,7)--(7,12);
\draw [thick] (0,10)--(1,10);
\draw [thick] (1,7)--(1,10);
   \filldraw (1,10) circle (2pt) node () {};
			\node[right] at (1,10) {\small $ \aa_4$};
           
   \filldraw (0,10) circle (2pt) node () {};
			\node[below] at (0,10) {\small $ b_{4}^{(4)}$};

   \filldraw (3,13) circle (2pt) node () {};
			\node[above] at (3,13) {\small  $\alpha_{5}$};
     \draw [thick] (3,13)--(0,10);

	\end{tikzpicture}

	\caption{A reduced basis for the complex $X^\infty_{5} (-T_{6,7}) \langle 5 \rangle$ where the coordinates are given by $\II$ and $\JJ$ filtrations. The generators are marked abstractly, without $U$ powers. The edges represent the differentials; the edge with $*$ depicts an instance of the fact that  $\Delta_{\II}(\alpha_n, b_{n}^{(n)}) =  \Delta_{\II}(\alpha_{n+1},b_{n}^{(n)}) + n $.}
	\label{fig:cn}
	\end{figure}

When $n=1$ and $2$, the local equivalence class of the complex $X^\infty_{n} (-T_{2n,2n+1})$ can be decided following a similar vein. We record the result in the next lemma, and the computations are left to the reader as an exercise. 

\begin{lemma} \label{le:case12}
    When $n=1$, the complex $X^\infty_{1} (-T_{2,3})$ is locally trivial.\\
    When $n=2$, the complex $X^\infty_{2} (-T_{4,5})$ has a local complex characterized by the following.
    \begin{align*}
        \partial \alpha_1 &= U^3 b_2^{(1)},\\
        \partial \alpha_2 &= U^2 b_2^{(1)} + U^3 b_2^{(2)},\\
        \partial \alpha_3 &= U b_2^{(1)};\\
        \Delta_{\II,\JJ}&(\alpha_1, b_2^{(1)}) = (3,1),\\
        \Delta_{\II,\JJ}&(\alpha_2, b_2^{(1)}) = (2,3),\\
        \Delta_{\II,\JJ}&(\alpha_2, b_2^{(2)}) = (3,2),\\
        \Delta_{\II,\JJ}&(\alpha_3, b_2^{(2)}) = (1,3).\\
    \end{align*}
\end{lemma}

\bibliographystyle{amsalpha}
\bibliography{bibliography}

\end{document}